\documentclass[10pt,twoside]{article}
\usepackage{amsfonts}
\usepackage{fancyhdr}
\usepackage{titlesec}
\usepackage{cite}
\usepackage{ifthen}
\usepackage{amssymb}
\usepackage{fancyhdr}
\usepackage{titlesec}
\usepackage[arrow,matrix]{xy}
\usepackage{pifont}
\usepackage{stmaryrd}
\usepackage{setspace}
\usepackage{indentfirst}
\usepackage{amsmath,amssymb,amscd,bbm,amsthm,mathrsfs,dsfont}
\input amssym.def

\titleformat{\section}{\centering\large\bfseries}{\S\arabic{section}}{1em}{}
\newboolean{first}
\setboolean{first}{true}

\newtheorem{theorem}{Theorem}[section]

\newtheorem{cor}{Corollary}[section]
\newtheorem{definition}{Definition}[section]
\begin{document}
\setlength\abovedisplayskip{2pt}
\setlength\abovedisplayshortskip{0pt}
\setlength\belowdisplayskip{2pt}
\setlength\belowdisplayshortskip{0pt}
\title{\bf \LARGE Remarks on a Paper by Leonetti and Siepe
\author {\small \textsc{GAO Hongya}  \quad \textsc{LIU Chao} \quad \textsc{TIAN Hong} \\
{\small College of Mathematics and Computer Science, Hebei
University, Baoding, 071002, P.R.China}}\date{}} \maketitle
\footnote{MR Subject Classification: 35J60, 35D30, 35J25.}
 \footnote{Supported by NSFC
(10971224) and NSF of Hebei Province (A2011201011).}

\begin{center}
\begin{minipage}{135mm}

{\bf \small Abstract}.\hskip 2mm {\small In 2012, F.Leonetti and
F.Siepe [1] considered solutions to boundary value problems of some
anisotropic elliptic equations of the type
$$
\left\{
\begin{array}{llll}
\sum\limits _{i=1}\limits^{n} D_i (a_i(x,Du(x)))=0, &x\in
\Omega,\\
u(x)=\theta (x), & x\in \partial \Omega.
\end{array}
\right.
$$
Under some suitable conditions, they obtained an integrability
result, which shows that, higher integrability of the boundary datum
$\theta$ forces solutions $u$ to have higher integrability as well.
In the present paper, we consider ${\cal K}_{\psi,
\theta}^{(p_i)}$-obstacle problems of the nonhomogeneous anisotropic
elliptic equations
$$
\sum_{i=1}^n D_i (a_i(x,Du(x)))=\sum_{i=1}^n D_i f^i(x).
$$
Under some controllable growth and monotonicity conditions. We
obtain an integrability result, which can be regarded as a
generalization of the result due to Leonetti and Siepe.}
\end{minipage}
\end{center}

\thispagestyle{fancyplain} \fancyhead{}
\fancyhead[L]{\textit{}\\
}\fancyfoot{}

\vspace{5mm}

\section{Introduction and statement of results}
Let $\Omega $ be a bounded open subset of $\mbox {R}^n$, $n\ge 2$.
For $p_k>1$, $k=1,2,\cdots, n$, we denote by $p_m$ and $\overline p$
the maximum value and the harmonic mean of $p_k$, respectively,
i.e.,
$$
p_m=\max\limits _{k=1,2,\cdots, n} p_k, \quad \overline p: \frac 1
{\overline p}=\frac 1 n \sum_{k=1}^n \frac 1 {p_k}.
$$

For $t>0$, the weak $L^t$-spaces, or Marcinkiewicz spaces,
$L_{weak}^t(\Omega)$, is defined (see [2, Chapter 1, Section 2] or
[3, Chapter 2, Section 5]) by all measurable functions $f$ such that
$$
|\{x\in \Omega: |f(x)|>\tau \}| \le \frac k {\tau ^t}
$$
for some positive constant $k=k(f)$ and every $\tau >0$, where $|E|$
is the $n$-dimensional Lebesgue measure of $E\subset \mbox {R}^n$.
We recall that if $f\in L_{weak}^t (\Omega)$ for some $t>1$, then
$f\in L^s (\Omega)$ for every $1\le s <t$.

The anisotropic Sobolev spaces $W^{1,(p_k)}(\Omega)$ and
$W_0^{1,(p_k)}(\Omega)$ are defined, respectively, by
$$
W^{1,(p_k)}(\Omega) =\left\{ v\in W^{1,1} (\Omega): D_kv\in
L^{p_k}(\Omega) \mbox { for every } k=1,2,\cdots, n\right \}
$$
and
$$
W_0^{1,(p_k)}(\Omega) =\left\{ v\in W_0^{1,1} (\Omega): D_kv\in
L^{p_k}(\Omega) \mbox { for every } k=1,2,\cdots, n\right \},
$$
where $D_kv=\frac {\partial v}{\partial x_k}$, $k=1,2,\cdots,n$.

Let us consider the following divergence elliptic equation
$$
\sum_{i=1}^n D_i (a_i(x,Du(x)))=\sum_{i=1}^n D_if^i(x), \eqno (1.1)
$$
and suppose that the Carath\'eodory functions $a_i(x, \xi) :\Omega
\times \mbox {R}^n \rightarrow \mbox {R}$, $i=1,2,\cdots, n$,
satisfy
$$
|a_i(x,z)| \le c_1\left(h(x)+\sum_{j=1}^n |z_j|^{p_j}
\right)^{1-1/p_i} \eqno(1.2)
$$
for almost every $x\in \Omega$, every $z=(z_1,z_2,\cdots, z_n)\in
\mbox {R}^n$ and any $i=1,2,\cdots, n$. Furthermore, there exists
$\tilde{\nu} \in (0,+\infty)$ such that
$$
\tilde {\nu} \sum_{i=1}^n |z_i-\tilde z_i|^{p_i} \le \sum_{i=1} ^n
(a_i(x,z)-a_i(x,\tilde z)) (z_i-\tilde z_i) \eqno(1.3)
$$
for almost every $x\in \Omega$ and any $z, \tilde z \in \mbox
{R}^n$. The integrability conditions for $f^i(x)$, $i=1,2,\cdots,
n$, in (1.1) and $h(x)\ge 0$ in (1.2) will be given later.

Let $\psi$ be any function in $\Omega$ with values in $\mbox {R}\cup
\{\pm \infty\}$ and $\theta \in W^{1,(p_k)}(\Omega)$. We introduce
$$
{\cal K} _{\psi,\theta}^{(p_k)}(\Omega) =\left\{ v\in
W^{1,(p_k)}(\Omega): v\ge \psi, \mbox { a.e. and } v\in \theta +
W_0^{1,(p_k)} (\Omega) \right\}.
$$
The function $\psi$ is an obstacle and $\theta$ determines the
boundary values.

\begin{definition}
A function $u\in \theta +W_0^{1,(p_k)} (\Omega)$ is called a
solution to the boundary value problem
$$
\left\{
\begin{array}{llll}
\sum\limits_{i=1}\limits^n D_i
(a_i(x,Du(x)))=\sum\limits_{i=1}\limits ^n D_if^i(x), &x\in
\Omega,\\
u(x)=\theta (x), & x\in \partial \Omega,
\end{array}
\right.  \eqno(1.4)
$$
if
$$
\int_\Omega \sum_{i=1}^n a_i(x,Du(x)) D_i\varphi (x) dx =
\int_\Omega \sum_{i=1}^n f^i(x) D_i\varphi (x)dx, \eqno(1.5)
$$
holds true for any $\varphi \in W_0^{1,(p_k)} (\Omega)$.
\end{definition}

\begin{definition}
A solution to the ${\cal K}_{\psi,\theta}^{(p_k)}$-obstacle problem
is a function $u\in {\cal K} _{\psi,\theta}^{(p_k)}(\Omega)$ such
that
$$
\int_\Omega \sum_{i=1}^n a_i(x,Du(x)) (D_iv(x)-D_iu(x)) dx \ge
\int_\Omega \sum_{i=1}^n f^i(x)(D_iv(x)-D_iu(x)) dx
$$
whenever $v\in {\cal K} _{\psi,\theta}^{(p_k)}(\Omega)$.
\end{definition}

In a recent paper [1], F.Leonetti and F.Siepe considered solutions
$u\in \theta+W_0^{1,(p_k)}(\Omega)$ to the boundary value problem
$$
\left\{
\begin{array}{llll}
\sum\limits_{i=1}\limits^n D_i (a_i(x,Du(x)))=0, &x\in
\Omega,  \\
u(x)=\theta (x), & x\in \partial \Omega
\end{array}
\right.\eqno(1.6)
$$
under the conditions
$$
|a_i(x,z)|\le c_2(1+|z_i|)^{p_i-1},\ \ i=1,2,\cdots,n \eqno(1.2)'
$$
and (1.3), and obtained an integrability result, which shows that,
higher integrability of the boundary datum $\theta$ forces solutions
$u$ to have higher integrability as well.

Note that the assumptions (1.2)$'$ and (1.3) are suggested by the
Euler equation of the anisotropic functional
$$
\int_\Omega (|D_1u_1|^{p_1}+|D_2u_2|^{p_2}+\cdots
+|D_nu_n|^{p_n})dx.
$$

Later, Gao, Zhang and Li [4] considered ${\cal
K}_{\psi,\theta}^{(p_i)}$-obstacle problems for the homogeneous
elliptic equations
$$
\sum_{i=1}^n D_i (a_i(x,Du(x)))=0, \ \ x\in \Omega   \eqno(1.7)
$$
under the conditions (1.2)$'$ and (1.3). A similar result was
obtained, which shows that, higher integrability of the datum
$\theta_*=\max \{\theta, \psi\}$ forces solutions $u$ to have higher
integrability as well.

Integrability property is important among the regularity theories of
nonlinear elliptic PDEs and systems, see [5-12]. In the present
paper, we consider ${\cal K}_{\psi, \theta}^{(p_i)}$-obstacle
problems of the nonhomogeneous anisotropic elliptic equations
$$
\sum_{i=1}^n D_i (a_i(x,Du(x)))=\sum_{i=1}^n D_i f^i(x), \ \ x\in
\Omega  \eqno(1.8)
$$
under the conditions (1.2) and (1.3) with suitable functions $h$ and
$f^i$, $i=1,2,\cdots,n$. The main result of this paper is the
following theorem.

\begin{theorem} Let $\theta\in W^{1,(q_k)} (\Omega)$, $q_k\in (p_k, +\infty)$,
$k=1,2,\cdots, n$, $0\le h \in L^{\tau}(\Omega)$ with
$\tau=\max\limits _{k=1,\cdots, n}\frac {q_k}{p_k}$, $\psi \in
[-\infty, +\infty]$ be such that $\theta_* =\max \{\psi, \theta\}
\in \theta + W_0^{1,(q_k)} (\Omega)$. Moreover $\overline p<n$. Then
for any solution $u\in {\cal K}_{\psi, \theta}^{(p_k)} (\Omega)$ to
the ${\cal K}_{\psi, \theta}^{(p_k)}$-obstacle problem, we have
$$
u\in \theta_* +L_{weak}^t (\Omega)
$$
provided that $f^i\in L^{p_i/(p_i-1-bp_i)}(\Omega)$, $i=1,\cdots,n$,
where
$$
t=\frac {\overline p ^*}{1-\frac {b\overline p^*}{\overline p} \frac
{p_m}{p_m-1}} >\overline p^*,  \eqno(1.9)
$$
$b$ is any number verifying
$$
0<b\le \min_{j=1,\cdots, n} \left(1-\frac
{p_j}{q_j}\right)\min_{i=1,\cdots, n} \left( 1-\frac 1 {p_i}\right)
\eqno(1.10)
$$
and
$$
b<\frac {p_m-1}{p_m} \frac {\overline p}{ \overline p^*}.
\eqno(1.11)
$$
\end{theorem}

The idea of the proof of Theorem 1.1 comes from [1]. Theorem 1.1 can
be regarded as a generalization of [1, Theorem 2.1]. The difficulty
in the proof of Theorem 1.1 is that, under the condition (1.2), we
need to derive that the constant $M$ in [1] is finite. To this aim,
we need to restrict the constant $b$ in Theorem 1.1 to satisfy
(1.10) instead of [1,(2.9)].

For solutions to boundary value problems (1.4), we have

\begin{theorem}
Let $\theta \in W^{1,(q_k)} (\Omega)$, $q_k\in (p_k, +\infty)$,
$k=1,2,\cdots, n$, $0\le h \in L^{\tau}(\Omega)$ with $\tau$ be as
in Theorem 1.1. Moreover $\overline p<n$. Then for any solution
$u\in \theta +W_0^{1,(p_k)} (\Omega)$ to the boundary value problem
(1.4), we have
$$
u\in \theta +L_{weak}^t (\Omega),
$$
provided that $f^i\in L^{p_i/(p_i-1-bp_i)}(\Omega)$, $i=1,\cdots,n$,
where $t$ verifies (1.9)-(1.11).
\end{theorem}

\begin{proof}
Take the obstacle function $\psi$ to be minus infinity in Theorem
1.1 we arrive at the desired result.
\end{proof}

When we are in the isotropic case, that is, $q_i=q>p=p_i$ for every
$i=1,2,\cdots,n$, we denote $W^{1,(p_i)}(\Omega)=W^{1,p}(\Omega)$,
$W_0^{1,(p_i)}(\Omega)=W_0^{1,p}(\Omega)$ and ${\cal
K}_{\psi,\theta} ^{(p_i)}(\Omega)= {\cal K}_{\psi,\theta}
^{p}(\Omega)$.

When $n>q_i=q>p=p_i$ for every $i=1,\cdots,n$, then
$$
\begin{array}{llll}
\displaystyle \min_{j=1,\cdots, n} \left(1-\frac
{p_j}{q_j}\right)\min_{i=1,\cdots, n} \left( 1-\frac 1
{p_i}\right)&= \left(1-\frac p q \right)\left( 1-\frac 1 p\right)\\
&\displaystyle <\frac {(p-1)(n-p)}{np}=\frac {p_m-1}{p_m} \frac
{\overline p}{\overline p^*};
\end{array}
$$
then we take
$$
b=\min\limits_{j=1,\cdots, n} \left(1-\frac
{p_j}{q_j}\right)\min\limits_{i=1,\cdots, n} \left( 1-\frac 1
{p_i}\right)
$$
and we get
$$
t=\frac {nq}{n-q}.
$$
Thus we have the following two corollaries.

\begin{cor} Let $\theta\in W^{1,q} (\Omega)$, $q\in (p, n)$, $0\le h \in
L^{q/p}(\Omega)$, $\psi \in [-\infty, +\infty]$ be such that
$\theta_* =\max \{\psi, \theta\} \in \theta + W_0^{1,q} (\Omega)$.
Then for any solution $u\in {\cal K}_{\psi, \theta}^{p} (\Omega)$ to
the ${\cal K}_{\psi, \theta}^{p}$-obstacle problem, we have
$$
u\in \theta_* +L_{weak}^t (\Omega)
$$
provided that $f^i\in L^{q/(p-1)}(\Omega)$, $i=1,\cdots,n$, where
$$
t=\frac {nq}{n-q}.
$$
\end{cor}

\begin{cor}
Let $\theta \in W^{1,q} (\Omega)$, $q\in (p,n)$,
 $0\le h \in L^{q/p}(\Omega)$. Then for any solution $u\in \theta +W_0^{1,p}(\Omega)$ to the boundary
value problem (1.4), we have
$$
u\in \theta +L_{weak}^t (\Omega),
$$
provided that $f^i\in L^{q/(p-1)}(\Omega)$, $i=1,\cdots,n$, where
$$
t=\frac {nq}{n-q}.
$$
\end{cor}

\section{Proof of Theorem 1.1}
\begin{proof} For $L>0$ and a function $w$, let $T_L(w)$ be the truncation of $w$
at level $L$, that is,
$$
T_L(w) =\left \{
\begin{array}{llll}
w, & |w| \le L,\\
\mbox {sgn} (w)L, & |w| >L.
\end{array}
\right.
$$
Let $u\in {\cal K} _{\psi,\theta}^{(p_k)}(\Omega)$ be a solution to
the ${\cal K}_{\psi,\theta}^{(p_k)}$-obstacle problem. If we take
$$
v=\theta_* +T_L (u-\theta_*)=\left\{
\begin{array}{llll}
\theta_*-L, & \mbox { for } u-\theta_* <-L,\\
u, & \mbox { for } -L\le u-\theta_* \le L,\\
\theta_*+L, &\mbox { for } u-\theta_* >L,
\end{array}
\right.
$$
then $v\in {\cal K}_{\psi, \theta} ^{(p_k)} (\Omega)$. Indeed, it is
obvious that $v\in W^{1,(p_k)} (\Omega)$; for the second and the
third cases of the above definition for $v$, we obviously have $v\ge
\psi$, and for the first case, $u-\theta_* <-L$, we have
$\theta_*>u+L \ge \psi +L$, this implies $v=\theta_*-L\ge \psi$; and
since $u\in \theta +W_0^{1,(p_k)}(\Omega)$ and $u\ge \psi$, then
$\theta_* =\max \{\psi, \theta\} =\theta =u$ on $\partial \Omega$,
thus $v=u=\theta $ on $\partial \Omega$, this implies $v\in \theta
+W_0^{1,(p_k)} (\Omega)$.

Definition 1.2 together with the definition of $v$ yields
$$
\begin{array}{llll}
&\displaystyle \int_{\{|u-\theta_*| >L\}} \sum_{i=1} ^n f^i\cdot
(D_i\theta_*-D_iu) dx\\
=&\displaystyle \int_{\Omega} \sum_{i=1} ^n f^i\cdot
(D_iv-D_iu) dx\\
\le &\displaystyle \int_{\Omega} \sum_{i=1} ^n
a_i(x,Du)\cdot (D_iv-D_iu) dx \\
=&\displaystyle \int_{\{|u-\theta_*| >L\}}\sum_{i=1} ^n
a_i(x,Du)\cdot (D_i\theta _* (x)-D_iu) dx.
\end{array} \eqno(2.1)
$$
Monotonicity (1.3) allows us to write
$$
\begin{array}{llll}
&\displaystyle \tilde {\nu} \sum_{i=1} ^n \int _{\{|u-\theta_*|
>L\}} |D_iu-D_i\theta _*| ^{p_i}dx\\
\le &\displaystyle \int_{\{|u-\theta_*| >L\}} \sum_{i=1}^n
(a_i(x,Du)-a_i(x,D\theta_*)) (D_iu-D_i\theta_*) dx,
\end{array}
$$
which together with (2.1) implies
$$
\begin{array}{llll}
&\displaystyle \tilde {\nu} \sum_{i=1} ^n \int _{\{|u-\theta_*|
>L\}} |D_iu-D_i\theta _*| ^{p_i}dx\\
\le &\displaystyle -\int _{\{|u-\theta_*|
>L\}}\sum_{i=1} ^n a_i(x,D\theta_*) (D_iu -D_i\theta_*) dx\\
&\displaystyle -\int_{\{|u-\theta_*| >L\}} \sum_{i=1} ^n f^i\cdot
(D_i\theta_* -D_iu) dx.
\end{array} \eqno(2.2)
$$
We now use anisotropic growth (1.2) and the H\"older inequality in
(2.2), obtaining that
$$
\begin{array}{llll}
&\displaystyle \tilde \nu \sum_{i=1}^n \int _{\{|u-\theta_*|
>L\}} |D_iu-D_i\theta_*|^{p_i}dx \\
\le &\displaystyle -\sum_{i=1} ^n  \int _{\{|u-\theta_*|
>L\}} a_i(x,D\theta_*) (D_iu-D_i \theta _*) dx \\
&\displaystyle -\sum_{i=1} ^n \int_{\{|u-\theta_*| >L\}} f^i\cdot
(D_i\theta_* -D_iu) dx\\
\le &\displaystyle c_1 \sum_{i=1} ^n  \int _{\{|u-\theta_*|
>L\}} \left( h+\sum_{j=1}^n|D_j\theta
_*|^{p_j}\right)^{1-1/p_i} |D_iu-D_i\theta_*| dx
\end{array}
$$
$$
\begin{array}{llll}
&\displaystyle +\sum_{i=1} ^n \int_{\{|u-\theta_*| >L\}} |f^i|
|D_i\theta_* -D_iu| dx\\
\le & \displaystyle  c_1\sum_{i=1}^n  \left( \int _{\{|u-\theta_*|
>L\}} h+\sum_{j=1}^n |D_j\theta _*|^{p_j}
dx\right) ^{1-1/p_i} \left(\int _{\{|u-\theta_*|
>L\}}|D_iu-D_i\theta_*|^{p_i}dx \right)^{1/p_i}\\
& \displaystyle +\sum_{i=1}^n \left(\int _{\{|u-\theta_*|
>L\}}|f^i|^{p_i/(p_i-1)}dx \right)^{1-1/p_i}
\left(\int _{\{|u-\theta_*|
>L\}}|D_iu-D_i\theta_*|^{p_i}dx \right)^{1/p_i}.
\end{array}
\eqno(2.3)
$$
Let $t_i$ be such that
$$
p_i<t_i \le q_i
$$
for every $i=1,\cdots, n$; $t_i$ will be chosen later. We use
H\"older inequality as follows
$$
\begin{array}{llll}
&\displaystyle \left( \int _{\{|u-\theta_*|
>L\}} h+\sum_{j=1}^n|D_j\theta _*|^{p_j}dx \right) ^{1-1/p_i}\\
\le &\displaystyle \left( \int _{\{|u-\theta_*|
>L\}} \left( h+\sum_{j=1}^n|D_j\theta _*|^{p_j}\right) ^{t_i/p_i}
dx\right)^{(p_i-1)/t_i} |\{ |u-\theta_*| >L \}|^{(t_i-p_i)(p_i-1)
/t_ip_i}.
\end{array} \eqno(2.4)
$$
We would like to choose $t_i$ such that the exponent
$$
b=\frac {(t_i-p_i)(p_i-1)}{t_ip_i}  \eqno(2.5)
$$
does not depend on $i$, and simultaneously
$$
M_1=\max _{i=1,\cdots, n} \left( \int _{\Omega} \left(
h+\sum_{j=1}^n|D_j\theta _*|^{p_j}\right) ^{t_i/p_i}
dx\right)^{(p_i-1)/t_i}  \eqno(2.6)
$$
is finite. To the first aim, we solve (2.5) with respect to $t_i$,
obtaining that
$$
t_i =\frac {p_i (p_i-1)}{p_i-1-bp_i}. \eqno(2.7)
$$
Since we need $t_i > 0$, we require that $p_i-1-bp_i>0$, that is
$b<1-\frac 1 {p_i}<1$; moreover, the limitation $p_i < t_i$ is
equivalent to $b> 0$. Finally, since we required $t_i\le q_i$, we
need
$$
0<b\le \left(1-\frac {p_i}{q_i}\right)  \left(1-\frac 1
{p_i}\right)<1-\frac 1 {p_i}<1, \ \mbox { for every } i=1,\cdots,n.
$$
To the second aim, since we need
$$
h+\sum_{j=1}^n|D_j\theta _*|^{p_j} \in L^{t_i/p_i}(\Omega),
$$
we then require
$$
\min _{j=1,\cdots, n} \frac {q_j}{p_j} \ge \frac {t_i}{p_i}, \ \mbox
{ for every } i=1,\cdots,n. \eqno(2.8)
$$
Note that (2.8) is equivalent to
$$
\min_{j=1,\cdots,n} \frac {q_j}{p_j} \ge \max _{i=1,\cdots,n} \frac
{t_i}{p_i}. \eqno(2.8)'
$$
We now show (2.8) occurs if
$$
b=\frac {(t_i-p_i)(p_i-1)}{t_ip_i} \le \min_{j=1,\cdots,n}
\left(1-\frac {p_j}{q_j}\right) \min_{i=1,\cdots, n}\left(1-\frac 1
{p_i}\right)  \eqno(2.9)
$$
holds true. In fact, from (2.9), for every $i=1,\cdots,n$, one has
$$
\left(1-\frac {p_i}{t_i}\right)\min_{i=1,\cdots,n} \left(1-\frac 1
{p_i}\right) \le b \le \min_{j=1,\cdots,n} \left(1-\frac
{p_j}{q_j}\right) \min_{i=1,\cdots, n}\left(1-\frac 1 {p_i}\right).
$$
This implies
$$
1-\frac {p_i}{t_i}\le \min_{j=1,\cdots,n} \left(1-\frac
{p_j}{q_j}\right)=1-\max _{j=1,\cdots,n}\frac {p_j}{q_j} , \ \ \mbox
{ for every } i=1,\cdots,n
$$
which is equivalent to (2.8).

Thus, for every $b$ such that (1.10) holds true, we can define $t_i$
as in (2.7) obtaining that $p_i < t_i \le q_i$, $b$ in (2.5) is
independent of $i$, and $M_1$ in (2.6) is finite.

Under the above assumptions on exponents, (2.4) becomes
$$
\begin{array}{llll}
&\displaystyle \left( \int _{\{|u-\theta_*|
>L\}} h+\sum_{j=1}^n|D_j\theta _*|^{p_j}dx \right) ^{1-1/p_i}\\
\le &\displaystyle \left( \int _{\{|u-\theta_*|
>L\}} \left( h+\sum_{j=1}^n|D_j\theta _*|^{p_j}\right) ^{t_i/p_i}
dx\right)^{(p_i-1)/t_i} |\{ |u-\theta_*| >L \}| ^{b}\\
\le&\displaystyle M_1 |\{|u-\theta_*| >L \}| ^{b},
\end{array}  \eqno(2.10)
$$
where $M_1$ be as in (2.6).

We use H\"older inequality again, obtaining that
$$
\begin{array}{llll}
&\displaystyle \left(\int _{\{|u-\theta_*|
>L\}}|f^i|^{p_i/(p_i-1)}dx \right)^{1-1/p_i}\\
\le &\displaystyle \left(\int _{\{|u-\theta_*|
>L\}}|f^i|^{t_i/(p_i-1)}dx \right)^{p_i(p_i-1)/t_ip_i}|\{|u-\theta_*| >L \}|
^{b}\\
\le &\displaystyle M_2 |\{|u-\theta_*| >L \}| ^{b},
\end{array}  \eqno(2.11)
$$
where
$$
M_2 =\max _{i=1,\cdots,n}\left(\int _{\Omega}|f^i|^{t_i/(p_i-1)}dx
\right)^{p_i(p_i-1)/t_ip_i}.
$$
From the assumption $f^i\in L^{p_i/(p_i-1-bp_i)}(\Omega)$,
$i=1,2,\cdots,n$ and (2.5), one has
$$
\frac {t_i}{p_i-1} =\frac {p_i}{p_i-1-\frac {(t_i-p_i)(p_i-1)}{t_i}}
=\frac {p_i}{p_i-1-bp_i}
$$
for every $i=1,2,\cdots,n$. This implies $M_2<\infty$.

If we insert (2.10) and (2.11) into (2.3), we easily get
$$
\begin{array}{llll}
&\displaystyle \tilde \nu \sum_{i=1}^n \int _{\{|u-\theta_*|
>L\}} |D_iu-D_i\theta_*|^{p_i}dx \\
\le &\displaystyle (c_1M_1+M_2)|\{|u-\theta_*| >L \}| ^{b}
\sum_{i=1}^n\left(\int _{\{|u-\theta_*|
>L\}}|D_iu-D_i\theta_*|^{p_i}dx \right)^{1/p_i},
\end{array} \eqno(2.12)
$$
which can be considered as an inequality of the following type
$$
\sum_{i=1}^n d_i \le \lambda \sum_{i=1}^n (d_i)^{1/p_i},
$$
with
$$
d_i=\int _{\{|u-\theta_*|
>L\}} |D_iu-D_i\theta_*|^{p_i}dx
$$
and
$$
\lambda =\frac {(c_1M_1+M_2)|\{|u-\theta_*| >L \}|
^{b}}{\tilde \nu}.
$$
Following the idea of [1], one can derive that
$$
u-\theta_*\in L_{weak}^{t}(\Omega)
$$
with $t$ verifies (1.9)-(1.11). This completes the proof of Theorem
1.1.
\end{proof}

\vspace{4mm}

\rm 
\end{document}